\documentclass[a4paper,10pt,leqno]{amsart}

\title{Degeneracies in quasi-categories}
\author{Wolfgang Steimle}
\address{Universit\"at Ausburg\\
               Institut f\"ur Mathematik\\
               Germany}
\email{wolfgang.steimle@math.uni-augsburg.de}
\date{February 22, 2018}

\usepackage[active]{srcltx}
\usepackage{pdfsync}
\usepackage{hyperref}
\usepackage{enumerate,amssymb, amsmath}
\usepackage{graphicx}
\usepackage{xcolor}
\usepackage[arrow,curve,matrix,tips,2cell]{xy}
  \SelectTips{eu}{10} \UseTips
  \UseAllTwocells

\DeclareMathAlphabet{\matheurm}{U}{eur}{m}{n}
\newcommand{\Set}{\matheurm{Set}}

\DeclareMathOperator{\id}{id}

  \newcommand{\cala}{\mathcal{A}}
  
  \newcommand{\calc}{\mathcal{C}}

\theoremstyle{plain}
\newtheorem{theorem}{Theorem}[section]

\newtheorem{lemma}[theorem]{Lemma}
\newtheorem{corollary}[theorem]{Corollary}

\newtheorem{addendum}[theorem]{Addendum}

\theoremstyle{definition}
\newtheorem{definition}[theorem]{Definition}

\theoremstyle{remark}

\newtheorem*{remark*}{Remark}
\newtheorem*{example*}{Example}
\newtheorem*{examples*}{Examples}
\newtheorem*{notation*}{Notation}

\newcommand{\arincl}{\ar@{^{(}->}}
\newcommand{\arinclinv}{\ar@{_{(}->}}

\newcommand{\op}{^{op}}

\newcommand{\Deltainj}{\Delta^{\operatorname{inj}}}

\newcounter{commentnr}

\begin{document}

\begin{abstract}
In this note we show that a semisimplicial set with the weak Kan condition admits a simplicial structure, provided any object allows an idempotent self-equivalence. Moreover, any two choices of simplicial structures give rise to equivalent quasi-categories. The method is purely combinatorial and extends to semisimplicial objects in other categories; in particular to semi-simplicial spaces satisfying the Segal condition (semi-Segal spaces). 
\end{abstract}

\maketitle

\section{Statement of the main results}

A semisimplicial set (called $\Delta$-set by \cite{RS}) is a functor $(\Deltainj)\op\to\Set$, where $\Deltainj$ is the category of the totally ordered finite sets $[n]=\{0,\dots,n\}$ and strictly monotone maps. Rourke--Sanderson \cite{RS} (see also \cite{McC}) showed that any semisimplicial set satisfying the Kan condition admits a simplicial structure. In this note we investigate under which condition there is a simplicial structure on a semisimplicial set which merely satisfies the ``weak Kan condition'' that all \emph{inner} horns can be filled. For notational simplicity we will refer to such an object as \emph{quasi-semicategory}. (Here ``semi'' stands for semisimplicial. We do not intend to suggest that such an object is, in general, a model for a non-unital infinity-category.)

Let $X$ be a semisimplicial set. For $f\in X_1$, we write $f\colon x\to y$ where $x=d_1f$ and $y=d_0 f$. If $f,g,h\in X_1$, we write $g\circ f\simeq h$ if there is a 2-simplex $\sigma$ such that $d_1\sigma = h$, $d_2\sigma=f$, and $d_0\sigma = g$. The symbol $\Delta^n$ will denote the semisimplicial $n$-simplex (\emph{i.e.}, the presheaf represented by $[n]$), and $\Lambda^n_i\subset \Delta^n$ the $(n,i)$-horn.

\begin{definition}
\begin{enumerate}
 \item $f\colon x\to x$ in $X$ is called \emph{idempotent} if $f\circ f\simeq f$ holds.
 \item A morphism $f\in X_1$ is called \emph{equivalence} if $f$ is both cartesian and cocartesian -- that is, if for any $n\geq 2$ there is a filler for any horn $\Lambda_n^n\to X$ whose last edge is $f$ and for any horn $\Lambda^n_0\to X$ whose first edge is $f$.
\end{enumerate}
\end{definition}

\begin{examples*}
\begin{enumerate}
\item\label{item:ex1} If $X$ is a quasi-category, by Joyal \cite{Joyal} this notion of equivalence agrees with the usual notion of equivalence (or quasi-isomorphism).
\item\label{item:ex2} Let $X=N(\calc)$ be the nerve of a non-unital category (so that $X$ is a quasi-semicategory). It is not hard to see that $f\colon x\to y$ is an equivalence in our sense if and only if for any object $z$, the maps
\[-\circ f\colon \calc(y,z)\to \calc(x,z)\quad\mathrm{and}\quad f\circ -\colon \calc(z,x)\to \calc(z,y)\]
are bijective. 
%\item\label{item:ex3} Generalizing the previous example, let $\calc$ be a non-unital simplicial category, such that all morphism simplicial sets $\calc(x,y)$ are Kan. Then its homotopy coherent nerve $\Nhc(\calc)$ is a quasi-semicategory.\footnote{Explicitly,
%\[\Nhc_k(\calc):=\Fun(\mathfrak C[\Delta^k]^o, \calc)\]
%where $\mathfrak C[\Delta^k]^o$ is obtained from the simplicial category $\mathfrak C[\Delta^k]$ as in \cite[1.1.5.1]{HTT} by deleting the identity morphisms. 
%The functor $\mathfrak C[-]^o$ extends to a colimit-preserving functor on the category of semisimplicial sets such that for any semisimplicial set $X$, we have $(\mathfrak C[X]^o)^+\cong \mathfrak C[X^+]$. It follows that $\mathfrak C[\Lambda^n_i]^o(i,j) \cong \mathfrak C[(\Lambda^n_i)^+](i,j)$ whenever $i\neq j$ so that the proof of \cite[1.1.5.10]{HTT} carries over to our situation.
%}
%
% Then a vertex $f$ in $\calc(x,y)$, viewed as a 1-simplex in $\Nhc(\calc)$, is an equivalence in the above sense if for any object $z$, the maps of simplicial sets
%\[-\circ f\colon \calc(y,z)\to \calc(x,z)\quad\mathrm{and}\quad f\circ -\colon \calc(z,x)\to \calc(z,y)\]
%are weak equivalences. (The proof is a variation of \cite[1.1.5.10]{HTT}.)
\end{enumerate}
\end{examples*}

%\begin{proof}[Sketch proof of the last statement]
%As in the proof of \cite[1.1.5.10]{HTT}, we have that the morphism simplicial sets $\mathfrak C[\Lambda^n_0](i,j)$ agree with those of $\mathfrak C[\Delta^n]$ unless $(i,j)=(1,n)$ and $(i,j)=(0,n)$. In the first case,
%\[\mathfrak C[\Lambda^n_0](1,n)\subset \mathfrak C[\Delta^n](1,n)=(\Delta^1)^{n-2}\]
%is the boundary $\dell(\Delta^1)^{n-2}$ of the cube while in the second case
%\[\mathfrak C[\Lambda^n_0](0,n)\subset \mathfrak C[\Delta^n](0,n)=(\Delta^1)^{n-1}\]
%consists of $\dell'(\Delta^1)^{n-1}$, the boundary with one face of the cube deleted. It follows that an extension of a functor $\mathfrak C[\Lambda^n_0]\to \calc$ to $\mathfrak C[\Delta^n]$ corresponds to an extension of a given commutative diagram
%\[\xymatrix{
%\dell(\Delta^1)^{n-2} \ar[d] \ar[rr] && \calc(y,z) \ar[d]^{-\circ f}\\
%\dell' (\Delta^1)^{n-1} \ar[rr] && \calc(x,z)
%}\]
%along the inclusion $(\dell'(\Delta^1)^{n-1}, \dell(\Delta^1)^{n-2})\subset (\Delta^1)^{n-1}, (\Delta^1)^{n-2})$. But such an extension is possible if the right vertical map is an equivalence. 
%
%Hence $f$ is cocartesian if the map $-\circ f$ is a weak equivalence. Dually $f$ is cartesian if the map $f\circ -$ is a weak equivalence. 
%\end{proof}

If $X$ is a quasi-category and $x\in X_0$, then the degeneracy $s_0(x)$ of $x$ has the property of being an idempotent equivalence of $x$. Our first result is a converse to this statement. We will say that a quasi-semicategory $X$ \emph{has a simplicial structure} if it is the underlying semisimplicial set of a simplicial set (which then is automatically a quasi-category).

\begin{theorem}\label{thm:existence}
Let $X$ be a quasi-semicategory and let $s_0\colon X_0\to X_1$ be any function such that for each $x\in X_0$, $s_0(x)$ is an idempotent equivalence $x\to x$. Then $X$ has  a simplicial structure whose degeneracy in degree $0$ coincides with $s_0$.
\end{theorem}

\begin{corollary}[Rourke-Sanderson]\label{cor:Kan}
Any semisimplicial set satisfying the Kan condition has a simplicial structure.
\end{corollary}

Theorem \ref{thm:existence} comes with a relative version, see Theorem \ref{thm:relative} below. From the relative version we will deduce:

\begin{theorem}\label{thm:uniqueness}
Let $\calc$, $\calc'$ be quasi-categories which have the same underlying semi\-simplicial set. Then $\calc$ and $\calc'$ are categorically equivalent.
\end{theorem}

In section \ref{sec:proofs} we will prove Theorems \ref{thm:existence} and \ref{thm:uniqueness} and deduce Corollary \ref{cor:Kan}. In section \ref{sec:generalization} we will generalize the results to semisimplicial objects in other categories.
%The relative existence theorem also yields results in the multisimplicial case. In section \ref{sec:multi} we spell this out for bi-semisimplicial sets. 

The results of this paper will be used by the author in the proof of an analog of Waldhausen's additivity theorem in the setup of cobordism categories \cite{Steimleforward}. The point is that cobordism categories are naturally categories without identities, just as cobordism \emph{spaces} (considered by Quinn, Ranicki, Laures--McClure and others) are naturally semi-simplicial sets, while it is usually more convenient to work with simplicial objects.

\section{The relative existence theorem}\label{sec:proofs}

We start by recalling some terminology. A semisimplicial map $p\colon X\to Y$ is called \emph{inner fibration} if any commutative diagram of semisimplicial sets
\begin{equation}\label{eq:lifting_diagram}
\xymatrix{
 \Lambda^n_i \ar[d] \ar[r]^h & X \ar[d]^p\\
 \Delta^n \ar[r]^k \ar@{.>}[ru]& Y
} 
\end{equation}
has a diagonal lift as dotted in the diagram, provided $0<i<n$. An element $a\in X_1$ is called \emph{$p$-cartesian} if any commutative diagram \eqref{eq:lifting_diagram} has a diagonal lift, provided $i=n$ and the last edge of $h$ is $a$; it is called $p$-cocartesian if it is $p\op$-cocartesian as an element of $X_1\op$. These definitions are in accordance with  the usual simplicial notions.
%(Here the symbols $\Lambda^n_i$ and $\Delta^n$ denote the semisimplicial set version of the $n$-simplex and the $(n,i)$-horn, of course.) 

If $Y$ has a simplicial structure, and $f\colon x\to x$ is a 1-simplex in $X$, then we call $f$ \emph{$p$-idempotent} if there is a 2-simplex $\sigma\in X_2$ all whose boundaries are $f$, and which projects to the degenerate simplex $s_0^2(p(x))\in Y_2$.

\begin{theorem}\label{thm:relative}
Let $p\colon X\to Y$ be an inner fibration of semisimplicial sets and $f\colon A\to X$ the inclusion of a semisimplicial subset; assume that $Y$ and $A$ have simplicial structures such that $p\circ f$ is a simplicial map. Let $s_0\colon X_0\to X_1$ be a map, compatible with the degeneracies $s_0$ on $A$ and $Y$, and such that for all $x\in X_0$, $s_0(x)$ is $p$-idempotent, $p$-cartesian, and $p$-cocartesian. 

Then $s_0\colon X_0\to X_1$ extends to a simplicial structure on $X$ such that $f$ and $p$ are simplicial.
\end{theorem}

\begin{addendum}\label{addendum}
If $p$ is a Kan fibration, then a map $s_0\colon X_0\to X_1$ as required in the Theorem exists always, so that a compatible simplicial structure on $X$ exists without further hypotheses.
\end{addendum}

Theorem \ref{thm:existence} is a special case of Theorem \ref{thm:relative} where $Y$ is the terminal object and $A=\emptyset$; Corollary \ref{cor:Kan} follows from the Addendum. The relative existence theorem also implies Theorem \ref{thm:uniqueness}: Let $J$ be the groupoid with two objects $0$ and $1$ and two non-identity morphisms. We apply Theorem \ref{thm:relative} with $X=\calc\times J$, $A=\calc\times \{0,1\}$, $Y=J$, and $p$ the projection map, where $A$ carries the simplicial structure of $\calc$ over $0$ and of $\calc'$ over $1$. 
%Note that for $x\in \calc_0$, the morphism $s_0(x)\in X_1$ is $p$-idempotent; it is  $p$-cocartesian and $p$-cartesian because it is an equivalence in the quasi-categories $\calc\times J$. Together with the same argument for $\calc'_0$, this shows that the assumptions of Theorem \ref{thm:relative} are satisfied.
We conclude that $\calc\times J$ has a simplicial structure compatible with $\calc$ over 0 and with $\calc'$ over 1, such that $p$ is simplicial. Now note that $p$ is a cartesian fibration over $J$ so the pull-backs over $0$ and $1$ are categorically equivalent \cite[3.3.1.3]{HTT}.

We come to the proof of Theorem \ref{thm:relative}, which is a modification of the strategy from \cite{McC}. Throughout this section $X$ and $A$ will be as in the assumption of Theorem \ref{thm:relative}. For notational brevity we will give the proof only in the case where $Y$ is the terminal object $\{*\}$ so that the datum of $p$ and $Y$ may be ignored. The proof in the general case is identical, if ``filling a horn'' is replaced by ``choosing a diagonal lift''.

Recall the simplicial identities:
% that a sequence of sets $X_0$, $X_1$, \dots, together with operators $d_i\colon X_n\to X_{n-1}$ ($0\leq i\leq n$) and operators $s_i\colon X_n\to X_{n+1}$ ($0\leq i\leq n$) is a simplicial set if and only if the following identities hold:
\begin{align}
\label{eq:d_vs_d}
d_i d_k &= d_{k-1} d_i \quad (i<k);
\\
\label{eq:d_vs_s}
d_i s_k &= 
\begin{cases}
s_{k-1} d_i  & (i<k),\\
\id & (i=k, k+1),\\
s_k d_{i-1} & (i>k+1);
\end{cases}
\\
\label{eq:s_vs_s}
s_i s_k &= s_{k+1} s_i \quad (i\leq k).
\end{align}

%START NEW APPROACH

The construction of degeneracy maps is by induction. Let us call an \emph{$N$-good system} a system of maps $(s_k\colon X_n\to X_{n+1})$ ($n\geq 0$, $0\leq k\leq \min(n,N)$) that satisfies the simplicial identities whenever they apply, and that extends the given maps on $A$ and $X_0$. Clearly a $(-1)$-good system exists; we wish to prove that any $(N-1)$-good system $(s_0, \dots, s_{N-1})$ can be extended to an $N$-good one. 

We proceed in two steps. Let us call an \emph{almost $N$-good system} a system of maps $(s_k\colon X_n\to X_{n+1})$ ($n\geq 0$, $0\leq k\leq \min(n,N)$) satisfying the condition for being $N$-good, except that we do not require the identity $d_{N+1}s_N=\id$ to hold. 

\begin{lemma}\label{lem:step1}
Any $(N-1)$-good system extends to an almost $N$-good system.
\end{lemma}

\begin{proof}
The construction of $s_N\colon X_n\to X_{n+1}$ is by induction on $n$, starting at $n=N$. In the case $N=0$, the induction beginning is provided by the map $s_0\colon X_0\to X_1$ which exists by assumption. The induction step and, in the case $N\neq 0$, also the induction beginning, are proven by the same construction which we now explain. 

Assume that we have an $(N-1)$-good system $(s_0, \dots, s_{N-1})$ and maps $s_N\colon X_\ell\to X_{\ell+1}$ for $ N\leq \ell\leq n-1$, satisfying the condition for being almost $N$-good whenever they apply. We wish to define $s_N\colon X_n\to X_{n+1}$ so that the conditions for being almost $N$-good hold whenever they apply; that is, \eqref{eq:d_vs_s} and \eqref{eq:s_vs_s} should hold for $k=N$ except we  do not require $d_{N+1}s_N=\id$. 

For $x\in X_n$, the equations in \eqref{eq:d_vs_s} for $k=N$, $i\neq N+1$, are $(n+1)$-many equations that together prescribe the restriction of $s_N(x)$ to the  horn $\Lambda^{n+1}_{N+1}\subset \Delta^{n+1}$. Therefore we will define $s_N(x)$ as a filler for the horn $\Lambda^{n+1}_{N+1}\to X$ which is defined by the right-hand sides of the relevant equations in \eqref{eq:d_vs_s}. In more detail, we let
\[x_i=\begin{cases}
       s_{N-1}d_i(x), &(i<N),\\
       x, & (i=N),\\
       s_N d_{i-1}(x), & (i>N+1)
      \end{cases}
\]
where the operator $s_N$ in the last case acts on $X_{n-1}$ and is given by hypothesis.  We claim that 
\begin{equation}\label{eq:horn_condition}
 d_j(x_i) = d_{i-1}(x_j), \quad (j<i,\; j,i\neq N+1)
\end{equation}
so that the sequence $x_i$ for $i\neq N+1$ defines a horn $\Lambda^{n+1}_{N+1}$ in $X$. The equations \eqref{eq:horn_condition} can be easily verified by hand using the relevant equations of \eqref{eq:d_vs_s}, making a case by case distinction.  
%(A less computational alternative argument will be sketched below.) 

If $n>N$ (the induction step case), the horn defined in this way is an inner horn so a filler exists because $X$ is a quasi-semicategory. If $n=N$ (the induction beginning case), this is a right horn, but applying \eqref{eq:d_vs_s} iteratively, we see that 
\[d_0^N s_N(x) = s_0 d_0^N(x)\]
so the last edge of the horn is in the image of $s_0\colon X_0\to X_1$ and therefore a cartesian morphism by our assumptions. So the horn has a filler in this case again, by definition of being cartesian.

We would like to define $s_N(x)$ as a choice of filler for this horn; however this definition is a little too crude in that we didn't ensure that the restriction of $s_N$ to $A_n$ is as required, nor that the simplicial identities \eqref{eq:s_vs_s} hold. This can be rectified as follows: First, if  $x=f(x')$ for some $x'\in A_n$, we have to and do choose $f s_N(x')$ as a filler for the horn in order to make $f$ simplicial. Second, if $x\in X_n$ is of the form $x=s_i(y)$ for some $i<N$, then the equation $s_Ns_i=s_is_{N-1}$ from \eqref{eq:s_vs_s} forces us to choose $s_N(x):=s_i s_{N-1}(y)$ as a filler for the horn. 
%(This is indeed a filler; this can again be verified by a case-by-case calculation or the more systematic argument below.) 

To complete the proof, we need to show this rule is well-defined; that is, if $x=s_i(y)=s_j(y')$ or if $x=f(x')=s_i(y)$, any of the choices leads to the the same value of $s_N(x)$. To justify this, we use the following Lemma, which we prove at the end of the section.

\begin{lemma}\label{lem:pullback_property}
In an $(N-1)$-good system, and for $i<j<N$, $k<N$, the commutative squares
\[\xymatrix{
X_{n-2} \ar[r]^{s_{j-1}} \ar[d]^{s_i} & X_{n-1} \ar[d]^{s_i} 
  && A_{n-1} \ar[r]^f \ar[d]^{s_k} & X_{n-1} \ar[d]^{s_k}\\
X_{n-1} \ar[r]^{s_j} & X_n
  && A_n \ar[r]^f & X_n
}\]
are pull-back squares.
%If $s_i(y)=s_j(y')$ with $i< j< N$, then $y=s_{j-1}(z)$ and $y'=s_i(z)$ for some $z$.
\end{lemma}

Hence, if we can write $x=s_i(y) = s_j(y')$ for $i<j<N$, there exists a $z\in X_{n-2}$ such that $y=s_{j-1}(z)$ and $y'=s_i(z)$. Then we have
\[s_i s_{N-1}(y) = s_i s_{N-1} s_{j-1}(z) = s_j s_{N-1} s_i(z)= s_j s_{N-1} (y')\]
provided $i<j<N$ and the system is $(N-1)$-good, so that both possible definitions of $s_N(x)$ agree. Similarly, if $x=s_i(y)=f(x')$, then there exists $y'\in A_{n-1}$ with $y=f(y')$ and $s_i(y')=x'$ so
\[s_i s_{N-1}(y) = s_i s_{N-1} f(y') = f s_i s_{N-1}(y') = f s_N s_i (y') = f s_N(x') \]
and again the two possible definitions agree. 
\end{proof}

Next we come to the second step of our construction.

\begin{lemma}\label{lem:step2}
If $(s_0, \dots, s_N)$ is an almost $N$-good system, then there is a collection of maps $\sigma_N\colon X_n\to X_{n+1}$, $n\geq N$, such that $(s_0, \dots, s_{N-1}, \sigma_N)$ is $N$-good.
\end{lemma}

\begin{proof}
We construct maps $T_N\colon X_n\to X_{n+2}$ for $n\geq N$ such that
\begin{align}
\label{eq:d_vs_T}
d_i T_N & =
\begin{cases}
s_{N-1}^2 d_i, & (i<N),\\
s_N, & (i=N+1, N+2),\\
T_N d_{i-2}, & (i> N+2);
\end{cases}
\\
\label{eq:s_vs_T}
T_N s_i & = s_N^2 s_i, & (i<N).
\end{align}

One should think of the map $T_N$ as a candidate for the double degeneracy $\sigma_N^2$. Indeed, if $s_N$ is already $N$-good, then the operators $T_N:= s_N^2$ satisfy the above properties (plus the equation $s_N=d_N T_N$). On the other hand, if we are given an almost $N$-good system $(s_0,\dots, s_N)$, and maps $T_N$ satisfying \eqref{eq:d_vs_T} and \eqref{eq:s_vs_T}, then by setting $\sigma_N:=d_NT_N$, we obtain an $N$-good system. 

The construction of the collection $(T_N)$ is very analogous to the construction in the previous step and is by induction on $n\geq N$. In the case $N=0$, the induction beginning is given by any map $T_0\colon X_0\to X_2$ that sends $x\in X_0$ to a 2-simplex expressing the fact that $s_0(z)$ is $p$-idempotent, where we also assume that on $A_0\subset X_0$, the map is actually given by $s_0^2$. The induction beginning in the other cases and the induction step are by the same construction as follows:

Assume that we have an almost $N$-good system $(s_0, \dots, s_N)$ and maps $T_N\colon X_\ell\to X_{\ell+2}$ for $ \ell\leq n-1$, satisfying the conditions \eqref{eq:d_vs_T} and \eqref{eq:s_vs_T}. We wish to define $T_N\colon X_n\to X_{n+2}$ so that \eqref{eq:d_vs_T} and \eqref{eq:s_vs_T} are again satisfied.

For $x\in X_n$, the ($n+2$ many) equations \eqref{eq:d_vs_T} define a map $\Lambda^{n+2}_N\to X$. In more detail, if we let 
\[x_i=\begin{cases}
       s_{N-1}^2d_i(x), &(i<N),\\
       s_N(x), & (i=N+1, N+2),\\
       T_N d_{i-2}(x), & (i>N+2)
      \end{cases}
\]
then again a case-by-case calculation shows that the horn equations 
\begin{equation}\label{eq:n_horn_condition}
 d_j(x_i) = d_{i-1}(x_j), \quad (j<i,\; j,i\neq N)
\end{equation}
hold. 
%(Again a more systematic argument will be given at the end of this section.)

If $N>0$, the horn $\Lambda^{n+2}_N$ is an inner horn which can be filled by an $(n+2)$-simplex we call $T_N(x)$. If $N=0$, then \eqref{eq:d_vs_T} shows that the first edge is $s_0$ of the first vertex, which is cocartesian by assumption. So we can fill in the horn as well to get an element $T_0(z)\in X_{n+2}$.

Again we need to modify this construction in two ways: First, if $x=f(x')$, we choose as filler the element $f s_N^2(x')$ provided by the simplicial set structure of $A$. Second, if $x\in X_n$ degenerate, then  the choice of filler $T_n(x)$ is forced to us by \eqref{eq:s_vs_T}. Again, Lemma \ref{lem:pullback_property} ensures that this is well-defined. 
\end{proof}

Lemmas \ref{lem:step1} and \ref{lem:step2} together prove the induction step and therefore Theorem \ref{thm:relative}. We now give the postponed Lemma \ref{lem:pullback_property}. It builds on the following Lemma, which is valid in an arbitrary category and whose proof is an easy exercise.

\begin{lemma}\label{lem:general_pullback_property}
Suppose that in the commutative square
\[\xymatrix{
A \ar[r]^i \ar[d]^f & B \ar[d]^g \\
X \ar[r]^j & Y
}\]
the morphism $i$ is a retract of $j$, and that $j$ is injective. Then the diagram is a pull-back diagram.
\end{lemma}

(Here, being a retract means that there exist morphisms $F\colon X\to A$ and $G\colon Y\to B$ such that $Ff=\id_A$, $Gg=\id_B$, and $iF=Gj$.) Lemma \ref{lem:pullback_property} follows from this result by choosing $d_i$ as vertical retractions.

\begin{proof}[Proof of the Addendum]
By the Kan condition for the horn $\Lambda^1_1\subset \Delta^1$, any $x\in X_0$ is $d_0$ of some 1-simplex $e$. Then filling in the $(2,2)$-horn
\[\xymatrix{
& x \ar[rd]^e\\
x \ar[rr]^e \ar@{.>}[ru]^f && y
}\]
yields an edge $f\colon x\to x$; filling in the $(3,0)$-horn
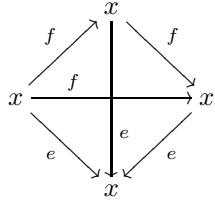
\[\xymatrix{
& x \ar[rd]^f \ar[dd]^(.7)e\\
x \ar[ru]^f \ar[rr]^(0.3)f \ar[rd]_e && x \ar[ld]^e\\
& x
}\] 
shows that $f$ is idempotent (and an equivalence, as any edge in a Kan semisimplicial set). Thus the correspondence $x\mapsto f$ provides a function $s_0$ as required.
\end{proof}

\section{A generalization}\label{sec:generalization}

The proof above works for semi-simplicial objects in other categories than the category of sets. Indeed, let $\calc$ be any category, closed under limits, and provided with a subclass of morphisms called ``cofibrations'', satisfying axioms $\mathbf{A}$ and $\mathbf{B}$ below.
\begin{description}
 \item[A] A split-injection is a cofibration.
\end{description}

For a collection of morphisms $X_i\to X$ ($i\in \{1,\dots, N\}$), their ``union'' $\bigcup_{i=1}^N X_i$ is defined to be the colimit of the objects $X_i$ over their ``intersections'' $X_{ij}:= X_i\times_X X_j$, that is, the colimit of the diagram formed by the objects $X_i$ and the objects $X_{ij}$ (for $i<j$), together with the projection maps $X_{ij}\to X_i$ and $X_{ij}\to X_j$. With this notation, the second axiom reads:

\begin{description}
 \item[B] If $(c_i\colon X_i\rightarrowtail X)_{i\in \{1,\dots, N\}}$ is a finite family of cofibrations, then their ``union'' exists and the  induced map $\bigcup_{i=1}^N X_i\to X$ is a cofibration. 
\end{description}

In the last section we studied the case where $\calc$ is the category of sets and the cofibrations are the injective maps. In this situation, one easily verifies that the ``union'' $\bigcup_n X_n$ maps injectively into $X$, with image the actual union of the subsets $c_n(X_n)\subset X$, which justifies our notation.

As usual, we call a morphism in  $\calc$  an ``acyclic fibration'' if it has the right lifting property against all cofibrations. Clearly the collection of acyclic fibrations is closed under compositions and pull-backs. In our previous example, the category of sets, a map is an acyclic fibration if and only if it is surjective.

Let $s\calc$ denote the category of semi-simplicial objects in $\calc$, and let $X\in s\calc$. Since $\calc$ is closed under limits, the contravariant functor $X$ extends along the Yoneda embedding $\Deltainj\to s\Set$, via the formula 
\[X(A):=\lim_{\Delta^n\to A} X_n \quad (A\in \Set^{(\Deltainj)\op})\]
where the limit is indexed over the category of simplices of $A$. With this definition, the canonical map $X_n\to X(\Delta^n)$ is an isomorphism.

\begin{definition}
Let $X,Y\in s\calc$ and $T\in \calc$.
\begin{enumerate}
\item A semisimplicial map $p\colon X\to Y$ is an \emph{inner fibration} (resp., a \emph{Kan fibration}) if the canonical maps 
\[X_n\to X(\Lambda^n_i)\times_{Y(\Lambda^n_i)} Y_n\]
in $\calc$ are acyclic fibrations for $0<i<n$ (resp., for $0\leq i\leq n$). 
\item Suppose further that $Y$ has a simplicial structure. A map $f\colon T\to X_1$ is \emph{$p$-idempotent} if there exists a map $T\to X_2$ which agrees with $f$ on all three boundaries, and whose image $T\to Y_2$ in $Y$ factors through the degeneracy $Y_0\to Y_2$. 
\item A map $f\colon T\to X_1$ is \emph{$p$-cartesian} if for any $n>0$, the canonical map
\[T\times_{X_1} X_n\to T\times_{X_1} X(\Lambda^n_n)\times_{Y(\Lambda^n_n)} Y_n\]
in $\calc$ is an acyclic fibration, where $X(\Lambda^n_n)$ maps to $X_1$ by the last edge map. The notion of $p$-cocartesianness is defined dually.
\end{enumerate}

\end{definition}

With these notions, we have the following generalizations of Theorem \ref{thm:relative} and Addendum \ref{addendum}.

\begin{theorem}\label{thm:general}
Let $p\colon X\to Y$ and $f\colon A\to X$ be morphisms in $s\calc$ where $p$ is an inner fibration and $f$ an injective cofibration in each semi-simplicial degree, and where $Y$ and $A$ have simplicial structures such that $p\circ f$ is simplicial. Let $s_0\colon X_0\to X_1$ be a $p$-idempotent, $p$-cartesian and $p$-cocartesion morphism in $\calc$ which is compatible with the degeneracies $s_0$ on $A$ and $Y$.

Then $s_0\colon X_0\to X_1$ extends to a simplicial structure on $X$ such that $f$ and $p$ are simplicial.
\end{theorem}

\begin{addendum}
If $p$ is a Kan fibration, then a map $s_0\colon X_0\to X_1$ as required in the Theorem exists always, so that a compatible simplicial structure on $X$ exists without further hypotheses.
\end{addendum}

The proof of Theorem \ref{thm:general} is identical to the proof of Theorem \ref{thm:relative}: In terms of this section, the proof of Lemma \ref{lem:step1} constructs a  commutative solid square
\[\xymatrix{
B \ar[r] \ar@{>->}[d] & X_{n+1} \ar[d]\\
X_n \ar[r] \ar@{.>}[ru]^{s_N} & X(\Lambda^{n+1}_{N+1})\times_{Y(\Lambda^{n+1}_{N+1})} Y_{n+1}
}\]
where $B$ is the ``union'' of the cofibrations $s_i\colon X_{n-1}\to X_{n}$, $i<N$, and the map $f\colon A_n\to X_{n}$; note that Lemma \ref{lem:pullback_property} precisely identifies the ``intersections'' of $s_i$ with $s_j$ and with $f$; and the calculation following the Lemma shows that the map $B\to X_{n+1}$ is well-defined. Therefore $s_N\colon X_n\to X_{n+1}$ exists by definition of inner fibration (in the case $n>N$) and by definition of $p$-cartesian (in the case $n=N$). By induction this gives rise to an almost $N$-good structure just as in the proof of Theorem \ref{thm:general}. The same re-writing can be made for Lemma \ref{lem:step2} and the second step of the proof. The proof of the Addendum is completely analogous. Here are two examples.

\subsection{Semi-Segal spaces}

We take $\calc$ the category of simplicial sets, with the usual notion of cofibration (level-wise injective maps). Then  a map is an acyclic fibration in our sense if and only if it is a Kan fibration and a weak equivalence (after realization). We call an object in $s\calc$ a semisimplicial space for short.

A map $p$ in $s\calc$ is a \emph{Reedy fibration} if for any inclusion $A\subset B$ of semi-simplicial sets, the induced map
\[X(B) \to Y(B)\times_{Y(A)} X(A)\]
is a Kan fibration. The space $X$ is called \emph{Reedy fibrant} if the projection $X\to \{*\}$ is a Reedy fibration. 

Let $I_n\subset \Delta^n$ be the semi-simplicial subset spanned by the edges $(i, i+1)$, where $i=0,\dots, n-1$. By definition, $X$ is a \emph{semi-Segal space} if it is Reedy fibrant and if for each $n>1$, the map $X_n\to X({I_n})$ induced by the inclusion $I_n\to \Delta^n$ is a weak equivalence of spaces. 
%For a semi-Segal space $X$, and ``objects'' $x,y\in X_0$, we define the ``morphism space'' $X(x,y)$ as the preimage of $(x,y)$ under the ``source-target map'' $(d_1,d_0)\colon X_1\to X_0$. The Segal condition implies that there is a composition map, well-defined up to homotopy
%\[-\circ -\colon X(y,z)\times X(x,y) \to  X(x,z).\]

The following is a variation of \cite[3.4]{Joyal-Tierney(2007)}.

\begin{lemma}
A Reedy fibration $p\colon X\to Y$ between semi-Segal spaces is an inner fibration in our sense. 
\end{lemma}

\begin{proof}
The map in question is a fibration by the fact that $p$ is a Reedy fibration. To show that it is a weak equivalence, we show that for each inner horn $\Lambda^n_k$, $0<k<n$, in the square
\[\xymatrix{
 X_n \ar[d] \ar[r]^p & Y_n \ar[d]\\
 X(\Lambda^n_k) \ar[r]^p & Y(\Lambda^n_k)
}\]
the vertical maps are weak equivalences. 

Recall that the forgetful functor from simplicial sets to semisimplicial sets has a left adjoint $A\mapsto A^+$ which is an embedding of categories. Let $\cala\subset \calc$ be the class of injective semi-simplicial maps (that is, injective simplicial maps that are of the form $f^+\colon A^+\to B^+$) such that $f^*\colon Y(B)\to Y(A)$ is a weak equivalence (hence an acyclic fibration). As $Y$ is Reedy fibrant by assumption, the  class $\cala$ contains the inclusion $L_n\to [n]$; by \cite[Lemma 3.5]{Joyal-Tierney(2007)} it contains therefore every inner horn inclusion $\Lambda^n_k\to [n]$, too. Thus $Y_n\to Y(\Lambda^n_k)$ is an acyclic fibration. The same argument applies to $X$.
\end{proof}

As a consequence, we deduce from Theorem \ref{thm:general} the following result. (Recall that here ``space'' means ``simplicial set''.)

\begin{theorem}\label{thm:relative_Segal}
Let $p\colon X\to Y$ be an inner fibration of semi-Segal spaces and $f\colon A\to X$ the inclusion of a semisimplicial subspace; assume that $Y$ and $A$ have simplicial structures such that $p\circ f$ is a simplicial map. Let $s_0\colon X_0\to X_1$ be a map, compatible with the degeneracies $s_0$ on $A$ and $Y$, and such that $s_0$ is $p$-idempotent, $p$-cartesian, and $p$-cocartesian. 

Then $s_0\colon X_0\to X_1$ extends to a simplicial structure on $X$ such that $f$ and $p$ are simplicial.
\end{theorem}

We close this section by giving a criterion for $p$-(co-)cartesianness. For $\sigma\in X(A)$ and $A\subset B$, we denote by $X(B)/\sigma\subset X(B)$ the subspaces of all elements mapping to  $\sigma\in X(A)$ under the map induced by the inclusion $A\subset B$.

\begin{lemma}\label{lem:criterion}
For a Reedy fibration $p\colon X\to Y$ of semi-Segal spaces, and $f\colon T\to X_1$, the following are equivalent:
\begin{enumerate}
\item $f$ is $p$-cartesian. 
\item For any $t\in T_0$, the composite $\{*\}\xrightarrow{t} T\xrightarrow{f} X_1$ is $p$-cartesian.
\item For any $t\in T_0$, with $e:=f(t)\colon x'\to x$, the following commutative square is a homotopy pull-back:
\[\xymatrix{
X_2/e \ar[rr]^{d_1} \ar[d]^p && X_1/x \ar[d]^p\\
Y_2/p(e) \ar[rr]^{d_1} && Y_1/p(x)
}\]
\end{enumerate}
\end{lemma}

\begin{remark*}
%Here we view $\Delta^0\subset \Delta^1$ as the inclusions of the last vertex and last edge. 
By the Segal condition, the map $d_2\colon X_2/e\to X_1/x'$ is a weak equivalence so the horizontal maps in the diagram may be thought of as ``postcomposition by $e$ and $p(e)$'', respectively.
\end{remark*}

\begin{proof}[Proof of Lemma \ref{lem:criterion}]
(i) implies (ii) because acyclic fibrations are stable under pull-back. For the converse direction, we note that in the map under consideration, 
\[T\times_{X_1} X_n\to T\times_{X_1} X(\Lambda^n_0)\times_{Y(\Lambda^n_0)} Y_n\]
both domain and target are Kan fibrations over $T$. Therefore, to show that the map is a weak equivalence, it suffices to test on all fibers over elements of $T_0$. But this is condition (ii).

For the equivalence between (ii) and (iii), we consider the following diagram (for $n>0$):
\begin{equation}\label{eq:comparison_for_criterion}
\xymatrix{
X_{n+1}/e \ar[d]^p \ar[r] & X(\Lambda^{n+1}_{n+1})/e \ar[d]^p \ar[r]^{d_n}&  X_n/y \ar[d]^p\\
Y_{n+1}/p(e) \ar[r] & Y(\Lambda^{n+1}_{n+1})/p(e)  \ar[r]^{d_n} & Y_n/p(y)
}
\end{equation}
and notice that condition (ii) is equivalent to the left square being a homotopy pull-back, for any $e=f(t)\in X_1$. Now we note that for $n=1$, the right horizontal maps are isomorphisms. Hence, if (ii) holds, then the total square is a homotopy pull-back for $n=1$, that is (iii) holds.

Conversely assume that (iii) holds. We consider the commutative diagram
\[\xymatrix{
X_{n+1}/e \ar[r]^{d_n} \ar[d]^\simeq & X_n/y \ar[d]^\simeq \\
X_{n-1}\times_{X_0} X_2/e \ar[r]^{\id\times d_1}_\simeq & X_{n-1}\times_{X_0} X_1/y
}\]
where the vertical arrows are equivalences by the Segal condition and the lower horizontal map is one by (iii). It follows that the total square in \eqref{eq:comparison_for_criterion} is a homotopy pull-back for all $n>0$. 

We show by induction on $n$ that the left square is a homotopy pull-back, too.  If $n=1$, then we remarked above that the right horizontal maps are isomorphisms which immediately implies the claim. For the induction step, we note that the inclusion $\Delta^n\subset \Lambda^{n+1}_0$ of the $n$-th face is obtained by filling in horns $\Lambda^k_0$ for $k\leq n$, with last edge $(n, n+1)$. (By induction on $k$, fill in all pairs of type $(i_1,\dots, i_k, n)$ and $(i_1, \dots, i_k, n+1)$; for each such pair this corresponds to filling in a horn as required.) By the inductive assumption, it follows that the right square is a homotopy pull-back, hence so is the left.
\end{proof}

\subsection{Multi-semisimplicial sets}

We take $\calc=s^k\Set$ the category of $k$-fold semisimplicial sets, where a morphism defined to be a cofibration if it is injective in each multi-semisimplicial level. We say that an object $X$ of $\calc$ satisfies the Kan condition if, after writing $s^k\Set= s(s^{k-1}\Set)$ by singling out any of the $k$ simplicial directions, any map
\[X_n \to X(\Lambda^n)\]
induced by a horn inclusion $\Lambda^n\subset \Delta^n$ is an acyclic fibration in the sense of this section. (This is equivalent to \cite[Definition 5.2]{McC}.)

\begin{theorem}[{\cite{McC}}]
Any $k$-fold semisimplicial set which satisfies the Kan condition, has a $k$-fold simplicial structure.
\end{theorem}

\begin{proof}
We show more generally that any $k$-fold semisimplicial $l$-fold simplicial set $X$, satisfying the Kan condition, has a $(k+l)$-fold simplicial structure. The proof is by induction on $k$, where the induction beginning $k=0$ holds obviously. For the induction step, we view $X$ as a semisimplicial object in the category of $(k-1)$-fold semisimplicial $l$-fold simplicial sets. By Theorem \ref{thm:general}, this can be promoted to a simplicial object, corresponding to a $(k-1)$-fold $(l+1)$-fold simplicial set. But this admits a simplicial structure by induction hypothesis.
\end{proof}

% 
% 
% \begin{thebibliography}{1}
% 
% \end{thebibliography}

\end{document}